\documentclass{amsart}
\usepackage{amsmath,amssymb}
\usepackage[UKenglish] {babel}

\usepackage{xy}
\xyoption{all}

\usepackage{amssymb,amsmath,amsthm,enumerate,graphicx}
\theoremstyle{definition}
\newtheorem{definition}{Definition}

\newtheorem{remark}{Remark}

\theoremstyle{plain}
\newtheorem{lemma}{Lemma}
\newtheorem{proposition}{Proposition}
\newtheorem{theorem}{Theorem}
\newtheorem{corollary}{Corollary}
\newtheorem{example}{Example}

\allowdisplaybreaks

\begin{document}

\title[Partially Alternative Real Division  Algebras]
{partially alternative real division  algebras with a few imaginary units}

\author{Tianran  Hua}

\address{T. Hua, Department Of Mathematical and
Computational Sciences, University of Toronto Mississauga, 3359
Mississauga Road N., Mississauga, On L5L 1C6}
\email{tianran.hua@mail.utoronto.ca}

\author{Marina  Tvalavadze}

\address{M. Tvalavadze, Department Of Mathematical and
Computational Sciences, University of Toronto Mississauga, 3359
Mississauga Road N., Mississauga, On L5L 1C6}
\email{marina.tvalavadze@utoronto.ca}

\begin{abstract}
In this paper, we extend the investigation of four-dimensional partially alternative algebras $\mathcal A$ initiated in \cite{HNT}. The partial alternativity condition, a natural generalization of the alternativity axiom, broadens the class of alternative algebras, enabling the exploration of non-associative structures with weakened algebraic constraints.

We focus on four-dimensional partially alternative unital real division algebras (RDAs) equipped with at least three distinct imaginary units and a reflection. Under these constraints, we achieve a complete classification of the isomorphism classes of such algebras and find that there are, in fact, infinitely many distinct classes.
This classification generalizes  the classical Frobenius and Hurwitz theorems in the four-dimensional setting. Moreover,  we  explicitly characterize  automorphism groups for these algebras showing that it is either $SO(3)$ if $\mathcal A\cong \mathbb H$ (the real quaternions) or $\mathbb Z_2$, otherwise. Since partially alternative algebras do not exist in odd dimensions, and the structure of two-dimensional partially alternative algebras is straightforward, we provide a classification of such algebras up to dimension four.
\end{abstract}
\maketitle


\section{Introduction}

The classical examples of finite-dimensional associative real division algebras are the real numbers $\mathbb{R}$, the complex numbers $\mathbb{C}$, and the quaternions $\mathbb{H}$, with dimensions 1, 2, and 4, respectively. Frobenius's theorem \cite{Frob}, established in 1878, demonstrates that these three algebras constitute a complete classification: every finite-dimensional associative division algebra over the real numbers is isomorphic to exactly one of $\mathbb{R}$, $\mathbb{C}$, or $\mathbb{H}.$ In his 1930 theorem \cite{Zorn}, Zorn demonstrated that relaxing the associativity requirement to alternativity introduces exactly one additional finite-dimensional real division algebra: the 8-dimensional octonion algebra $\mathbb O$. This result completes the classification of real alternative division algebras, which are limited to $\mathbb R$ (reals), $\mathbb C$ (complexes), $\mathbb H$ (quaternions), and $\mathbb O$ (octonions).

Finite-dimensional real algebras containing the algebra of complex numbers in a specified manner have already attracted significant interest from researchers. For example, in \cite{MPS}, the authors investigate \emph{locally-complex} algebras defined by the property that the subalgebra generated by any element not a scalar multiple of the unit 1 is isomorphic to $\mathbb C$.
Their work provides a complete classification of such algebras in dimensions up to 4. 

In \cite{HNT}, the \emph{partial alternativity property}—a natural generalization of alternativity for nonassociative algebras—was introduced for the first time. These are unital algebras that contain a subalgebra isomorphic to the complex numbers, with respect to which they are $\mathbb{C}$-bimodules. As shown in the same work, such algebras exist in even dimensions. Here, we demonstrate that they do not exist in odd dimensions.

Let us recall the rigorous definition of \emph{partially alternative} algebras.

\begin{definition} \label{ImaginaryUnits} 
Let $\mathcal{A}$ be a real nonassociative algebra with a unit element $1$. An element $q \in \mathcal{A}$ is called an \textit{imaginary unit}  if $q^2=-1$.
Denote by $\mathcal{I}_\mathcal{A}$  the set of all imaginary units in $\mathcal{A}$.
\end{definition}

\begin{definition}\label{PA}
Let $\mathcal{A}$ be a real  nonassociative algebra with a unit element $1$ and $\mathcal{I}_\mathcal{A}\neq \emptyset$ (i.e. the set is nonempty).  Let  $(a, b, c)= (ab)c - a(bc)$ for $a,b,c \in \mathcal A$.  Then\\
(1)  $\mathcal{A}$ is called \textit{partially left alternative}  if  for all $x\in\mathcal{I}_\mathcal{A}$ and $y\in\mathcal{A}$,  $(x,x,y)=0$;\\
(2) $\mathcal{A}$ is called \textit{partially flexible} if  for all $x\in\mathcal{I}_\mathcal{A}$ and $y\in\mathcal{A}$, $(x,y,x)=0$;\\
(3) $\mathcal{A}$ is called \textit{partially right alternative} if for all $x\in\mathcal{I}_\mathcal{A}$ and $y\in\mathcal{A}$, $(y,x,x)=0$.

The algebra $\mathcal{A}$ is called \textit{partially alternative}  if it is partially left alternative, flexible, and right alternative.
\end{definition}

Assuming a four-dimensional partially alternative real division algebra possesses a unit and a reflection (i.e., a nontrivial order-2 automorphism), the following result was established.

\begin{theorem}\cite{HNT} Let $\mathcal A$ be a four-dimensional partially alternative real division algebra with the unity $1$.  Assume that $\mathcal A$ admits a reflection $\varphi$.  Then there exists a basis $\{1, i, w, v\}$ in $\mathcal A$ with 
the following multiplication table $(T_p)$:

$$ \begin{tabular}{c|cccc} 
          & $\bf 1$            & $\bf i$                                  & $\bf w$                                                 & $\bf v$      \\ 
\hline
$\bf 1$     &  $1$           & $i$                                  &  $w$                                                 &   $v$     \\
$\bf i$      &  $i$            & $-1$                                &  $-v$                                                &   $w $ \\
$\bf w$     &  $w$          & $v$                                 &  $ a 1 + b i$                                       & $ c 1 + d i$   \\ 
$\bf v$     &  $v$           &  $-w$                              &   $e 1 + f i $                                       & $g 1 +  h i$  \\ 
\end{tabular} \eqno(T_p) $$
where $a, b, c, d, e, f, g, h \in\mathbb R$.
\end{theorem} 
\par\medskip

This paper extends prior work (\cite{HNT}) by classifying four-dimensional unital partially alternative real division algebras equipped with a reflection and at least three imaginary units. 
 By relaxing the requirement of alternativity to partial alternativity, we obtain new examples of real division algebras in dimension four, and moreover, demonstrate the existence of infinitely many isomorphism classes of such algebras.

\section{A General Framework for Partial Alternative Algebras}
In \cite{HNT}, the authors established the existence of partially alternative real algebras in every even dimension. The present work nonetheless remains primarily oriented toward the applications of partial alternativity to the classification of four-dimensional real division algebras. Notwithstanding our primary focus, we first provide a concise exposition of the general theory of partial alternativity.
\begin{proposition}
Let $\mathcal{A}$ be a real algebra with a unit $1\in\mathcal{A}$ and a set of imaginary units $\mathcal{I_A}$. Denote by $L_x:\mathcal{A}\rightarrow\mathcal{A}$ and $R_x:\mathcal{A}\rightarrow\mathcal{A}$ the operator of left multiplication by $x$ and the operator of right multiplication by $x$, respectively. Then,
\begin{itemize}
    \item[(1)] $\mathcal{A}$ is partially left alternative if and only if for every imaginary unit $i\in\mathcal{I_A}$, the operator $L_i$ of left multiplication is a linear complex structure on the vector space $\mathcal{A}$;
    \item[(2)] $\mathcal{A}$ is partially flexible if and only if for every imaginary unit $i\in\mathcal{I_A}$, the operators $L_i$ and $R_i$ commute;
    \item[(3)] $\mathcal{A}$ is partially right alternative if and only if for every imaginary unit $i\in\mathcal{I_A}$, the operator $R_i$ of right multiplication is a linear complex structure on the vector space $\mathcal{A}$.
\end{itemize}
\end{proposition}
\begin{proof}
Fix an imaginary unit $i\in\mathcal{I_A}$. Then, for all $x\in\mathcal{A}$, the associator $(i,i,x)$ vanishes if and only if $(L_i^2)(x)=i(ix)=(ii)x=-x=(-\operatorname{id}_\mathcal{A})(x)$ which is equivalent to $L_i$ being a linear complex structure on $\mathcal{A}$. The proof for the other two cases of partial flexibility and partial right alternativity follows analogously.
\end{proof}
\begin{corollary}
Let $\mathcal{A}$ be a real algebra. If $\mathcal{A}$ is partially left alternative or partially right alternative, then $\mathcal{A}$ is an even-dimensional real vector space.
\end{corollary}
\begin{proof}
Let $\mathcal{A}$ be partially left alternative or partially right alternative. Then, $\mathcal{A}$ admits a linear complex structure $J:\mathcal{A}\rightarrow\mathcal{A}$ whose minimal polynomial is the quadratic polynomial $\lambda^2+1$. By the use of rational canonical forms, the characteristic polynomial of $J$ is of the form $(\lambda^2+1)^n$. It follows that $\mathcal{A}$ is even-dimensional.
\end{proof}
We recall the well-known fact that any two of the three properties\textemdash left alternativity, flexibility, and right alternativity\textemdash jointly imply the third, and hence alternativity \cite{Schafer}. It remains an open question whether this implication extends to the setting of partial alternativity, and no counterexamples are currently known.

It is important to emphasise, however, that when considered in isolation, partial flexibility is strictly \textit{weaker} than either partial left or partial right alternativity. In particular, it imposes no restriction on the dimension of the algebra.
\begin{example}
There exists an odd-dimensional partially flexible algebra. 
\end{example}
\begin{proof}
$$ \begin{tabular}{c|ccc} 
& $\bf 1$ & $\bf i$ & $\bf x$ \\ 
\hline
$\bf 1$ & $1$ & $i$ & $x$  \\
$\bf i$ & $i$ & $-1$ & $i$ \\
$\bf x$ & $x$ & $i$ & $x$ 
\end{tabular}$$
The three-dimensional algebra given by the above multiplication table is partially flexible. It suffices to check that the algebra admits exactly two imaginary units.
\end{proof}

We write $(\mathcal{A},J)$ for the complex vector space whose underlying real vector space is $\mathcal{A}$ and which is endowed with the linear complex structure $J:\mathcal{A}\to\mathcal{A}$; we denote by $\operatorname{span}_{J}(S)$ the complex span of any subset $S\subseteq\mathcal{A}$ in $(\mathcal{A},J)$.

The fact that the operators $L_i$ and $R_i$ commute implies that each is complex-linear with respect to the complex structure induced by the other; that is, for every imaginary unit $i\in\mathcal{I_A}$, the operators $L_i:(\mathcal{A},R_i)\rightarrow(\mathcal{A},R_i)$ and $R_i:(\mathcal{A},L_i)\rightarrow(\mathcal{A},L_i)$ are complex-linear. This follows directly from the definitions.

Combined with partial left and right alternativity, the mere commutativity of the left and right multiplication operators $L_i$ and $R_i$ allows us to establish a few stronger results.

\begin{proposition}
Let $\mathcal{A}$ be a partially alternative real algebra. Then, for all imaginary units $i\in\mathcal{I_A}$, the linear endomorphism $x\mapsto ixi$ on $\mathcal{A}$ (equivalently, the adjoint action $\operatorname{ad}_i:=x\mapsto ix(-i)$ which differs by a sign) is an involution, by virtue of the partial flexibility identity.
\end{proposition}
\begin{proof}
Fix an imaginary unit $i\in\mathcal{I_A}$. The map $x\mapsto ixi$ is given by the composition $L_iR_i$, and, hence, it defines a linear endomorphism on $\mathcal{A}$. Due to associativity and commutativity of the operators $L_i$ and $R_i$, we observe that
$$(L_iR_i)^2=L_i(R_iL_i)R_i=L_i^2R_i^2=\operatorname{id}_{\mathcal{A}}.$$
Thus, $x\mapsto ixi$ is an involution.
\end{proof}
\begin{corollary}\label{ESD}
Let $\mathcal{A}$ be a partially alternative real algebra. For an imaginary unit $i\in\mathcal{I_A}$, there exists an eigenspace decomposition of $\mathcal{A}$ of the form $$\mathcal{A}=\mathcal{E}_{+1}(i)\oplus\mathcal{E}_{-1}(i)$$
where $\mathcal{E}_{+1}(i)=\{x\mid ixi=x\}$ and $\mathcal{E}_{-1}(i)=\{x\mid ixi=-x\}$.
\end{corollary}
\begin{proof}
The eigenspace decomposition follows from the fact that the composition $L_iR_i$ defines a linear involution on $\mathcal{A}$ and its spectrum is contained in the set $\{\pm1\}$. For simplicity, we denote the eigenspaces corresponding to the eigenvalues $+1$ and $-1$ by $\mathcal{E}_{+1}(i)$ and $\mathcal{E}_{-1}(i)$, respectively.
\end{proof}

\begin{corollary}
Let $\mathcal{A}$ be a partially alternative real algebra. Every element $x\in\mathcal{A}$ admits a unique decomposition $x=x_+ +x_-$ so that $x_+=ix_+i=\dfrac{x+ixi}{2}$ and $x_-=-ix_-i=\dfrac{x-ixi}{2}$. Alternatively, $\dfrac{\operatorname{id}-L_iR_i}{2}$ and $\dfrac{\operatorname{id}+L_iR_i}{2}$ are projections onto $\mathcal{E}_{+1}(i)$ and $\mathcal{E}_{-1}(i)$, respectively.
\end{corollary}
\begin{proof}
Define $x_+$ and $x_-$ as claimed in the corollary. It is straightforward to check that $x_+\in\mathcal{E}_{+1}(i)$ and $x_-\in\mathcal{E}_{-1}(i)$, and that $x$ is their sum. The uniqueness follows from the properties of the eigenspace decomposition.
\end{proof}



As is customary, we denote the commutator and anticommutator of $a,b$ by $[a,b]:=ab-ba$ and $\{a,b\}:=ab+ba$, respectively. The preceding results can also be expressed as the following corollary.

\begin{corollary}
Let $\mathcal{A}$ be a partially alternative real algebra and write $x=x_+ +x_-$ where $x_+$ and $x_-$ are defined above. Then, $x_+$ anticommutes with $i$ and $x_-$ commutes with $i$. Specifically, the following identities hold:
$$[i,x]=-[x,i]=2ix_+=-2x_+i,$$
$$\{i,x\}=2ix_-=2x_-i.$$
\end{corollary}
\begin{proof}
These identities can be verified by direct computations.
\end{proof}
\begin{remark}
Since $i$ commutes with  itself,  the complex subalgebra $\mathcal C\cong\operatorname{span}_{\mathbb{R}}(\{1,i\})$ is always a subset of $\mathcal{E}_{-1}(i)$.
\end{remark}
\begin{corollary}
Let $\mathcal{A}$ be a partially alternative real algebra. For every nonzero element $x\in\mathcal{A}$, $\operatorname{span}_{\mathbb R}(\{x,ix,xi,ixi\})$ is a real subspace of $\mathcal{A}$ that is $L_i$-invariant and $R_i$-invariant that is also a complex subspace of both $(\mathcal{A},L_i)$ and $(\mathcal{A}, R_i)$. 
The subspace is two-dimensional if and only if $x\in\mathcal{E}_{\pm1}(i)$, that is, $x$ is an eigenvector.
\end{corollary}
\begin{proof}

Let us assume that $x$ is an eigenvector, that is, $x\in\mathcal{E}_{\pm1}(i)$.  Then $ixi={\pm}x$, and, in particular, 
$\operatorname{span}_{\mathbb R}(\{x,ix,xi,ixi\}) = \operatorname{span}_{\mathbb R}(\{x,ix,xi\}).$ By the previous corollary, 
$[i, x] = ix - xi = 2i x_{+}$ and $\{i, x\} = ix +xi = 2x_{-} i.$ 

If $x \in \mathcal{E}_{+1}(i)$, then $x_{-} = 0$ and $x_{+} = x$.  Hence, 
$$ xi = ix - 2ix_{+} = ix_{+} - 2i x_{+} = -ix_{+} = -ix.$$
It follows that $\operatorname{span}_{\mathbb R}(\{x,ix,xi,ixi\}) = \operatorname{span}_{\mathbb R}(\{x,ix\}).$ Moreover, 
elements $x$ and $ix$ are linearly independent. This implies that $\dim\,\operatorname{span}_{\mathbb R}(\{x,ix,xi,ixi\}) = 2$. The case when $x \in \mathcal{E}_{-1}(i)$ is 
analogous.

Let us now assume that $\dim\,\operatorname{span}_{\mathbb R}(\{x,ix,xi,ixi\}) = 2$.  Since $x$ and $ix$ are clearly linearly independent, they form a basis for this span.  Assume the
contrary, that is, $x \notin \mathcal{E}_{\pm1}(i)$. Then both $x_{+} \neq 0$ and $x_{-} \neq 0$.  Next consider $ixi \in \operatorname{span}_{\mathbb R}(\{x,ix,xi,ixi\})$. Given that
$ixi = a x + b ix$ where $a, b\in \mathbb R$. Write $x = x_{+} + x_{-}$. Then $ixi = x_{+} - x_{-}$. It follows that 
$$ x_{+} - x_{-} = ax_{+} + ax_{-} + bi x_{+} + bi x_{-} = (a+bi)x_{+} + (a+bi) x_{-}.$$ This leads to a contradiction, $a+bi = \pm 1$, impossible.
Thus, either $x_{+}$ or $x_{-}$ is zero, as needed.


\end{proof}
\begin{remark}
The preceding corollary is a quantitative version of the usual result that  $\mathcal{E}_{\pm1}(i)$ are complex subspaces of $(\mathcal{A},L_i)$ and $(\mathcal{A},R_i)$ that are $L_i$-invariant and $R_i$-invariant. In particular, their real dimensions must be even.
\end{remark}

The prior results in this section have been derived with minimal use of abstract machinery. Nonetheless, a similar approach can also be employed to obtain the results concerning the eigenspace decomposition $\mathcal{E}_{+1}(i)\oplus \mathcal{E}_{-1}(i)$ and its subsequent developments. 

In what follows, we provide a brief survey of an auxiliary but equally effective approach. In \cite{HNT}, the authors notice that a partially alternative real algebra $\mathcal{A}$ is always a $\mathcal{C}$-bimodule  where $\mathcal{C}$ denotes a subalgebra generated by a fixed imaginary unit $i\in\mathcal{I_A}$ (see also \cite{AHK}).

Let us now recall some basic constructions from the theory of rings and algebras. Let $R$ be a commutative unital ring and let $C$ be an $R$-algebra. Multiplication algebras $\langle L_r,R_r\rangle\subseteq\operatorname{End}_R(C)$  form a well-developed theory that has been explored both abstractly and in concrete settings such as Lie, Jordan, and other nonassociative algebras (see \cite{Schafer, KS}).

If $C$ is an associative unital $R$-algebra, then the universal enveloping algebra  $\mathcal{E}(C) = C \otimes_R C^{\mathrm{op}}$.  An $R$-module $A$ is a $C$-bimodule if and only if $A$ is a left $\mathcal E(C)$-module. The naturality of this identification is reflected in the categorical equivalence $$R-_C\mathbf{Bimod}_C\simeq R-\mathbf{Mod}_{\mathcal{E}(C)}.$$

We now turn to the field of complex numbers, $A=\mathbb{C}$, viewed simultaneously as the separable algebraic closure of 
$K=\mathbb R$ and a quadratic étale real algebra. It follows from a diagonalisation result (for example, see subsection V.6.3 Algèbres diagonalisables et algèbres étales in \cite{Bourbaki}) that if $A$ is a finite étale algebra over a perfect field $K$, then there exists an isomorphism
\[\iota:\mathbb{C}\otimes_{\mathbb R}\mathbb C=\overline{K}\otimes_K A\xrightarrow{\sim}\mathbb{C}\times \mathbb C=\prod_{\varphi\in\operatorname{Hom}_{K-\mathbf{Alg}}(A,\overline{K})}\overline{K}\]
defined on elementary tensors by $\iota(k\otimes a)=(k\varphi(a))_{\varphi\in\operatorname{Hom}(A,\overline{K})}$.
\par\medskip
In our case,  the only $\mathbb R$-algebra homomorphisms of $\mathbb C$ into itself are the identity and the complex conjugation. Under the assumption that the identity mapping appears in the first coordinate while complex conjugation appears in the second coordinate, the explicit formula for an isomorphism is $\iota(z\otimes w)=(zw,z\overline{w})$. Furthermore, the product algebra admits two primitive central idempotents $(1,0)$ and $(0,1)$ whose preimage under the isomorphism $\iota$ are $e_-:=\dfrac{1\otimes 1-i\otimes i}{2}$ and $e_+:=\dfrac{1\otimes 1+i\otimes i}{2}$, respectively. This implies that a left $\mathcal{E}(\mathbb C)$-module $A$ is always a direct sum $e_-A\oplus e_+A$.

We now recall the relevant subalgebra $\mathcal{C}$ of $\mathcal A$ is generated by a fixed imaginary unit $i\in\mathcal{I_A}$. Since the partially alternative algebra $\mathcal{A}$ is a left $\mathcal{C}\otimes_{\mathbb R}\mathcal{C}$-module, we define a real algebra homomorphism $\Phi:\mathcal C\otimes_{\mathbb R}\mathcal C\rightarrow\operatorname{End}_{\mathbb R}(\mathcal{A})$ by the formula $$\Phi((a+bi)\otimes (c+di))=(a+bL_i)(c+dR_i)$$ using the left and right linear complex structures $(\mathcal{A},L_i)$ and $(\mathcal{A},R_i)$.

Define the commutator tensor $k=i\otimes 1-1\otimes i$ and the anticommutator tensor $a=i\otimes 1+1\otimes i$. A straightforward computation shows that $e_-k=ke_-=0$ and $e_+a=ae_+=0$, from which the operator identities $(\Phi(e_-))(L_i-R_i)=0$ and $(\Phi(e_+))(L_i+R_i)=0$ hold. Consequently, every element $x\in e_-\mathcal{A}$ commutes with $i$, while every $x\in e_+\mathcal{A}$ anticommutes with $i$, recovering the previously established eigenspace decomposition. Also, note 
that $\Phi(e_-)$ and $\Phi(e_+)$ are exactly the respective eigenprojections.
\par\medskip

Finally, we present our conclusions on commutativity and anticommutativity in the following concise multiplication-tabular form.
\begin{proposition}\label{GeneralPA}
Let $\mathcal{A}$ be a partially alternative unital real algebra of dimension $2n>2$, and let $i\in\mathcal{I_A}$ be a fixed imaginary unit. Suppose $\mathcal{E}_-(i)$ is $(2m+2)$-dimensional and $\mathcal{E}_+(i)$ is $(2n-2m-2)$-dimensional. 

Then, one may choose a basis $\{1,i,x_1,ix_1,\dots,x_m,ix_m\}$ for $\mathcal{E}_-(i)$ and a basis $\{x_{m+1},ix_{m+1},\dots,x_{n-1},ix_{n-1}\}$ for $\mathcal{E}_+(i)$. Moreover, with respect to any index $1\leq p\leq n-1$, the multiplication table on $\{1,i,x_p,ix_p\}$ takes the following form:
$$\begin{tabular}{c|cccc} 
& $\bf 1$ & $\bf i$ & $\bf x_p$ & $\bf ix_p$ \\ 
\hline
$\bf 1$ & $1$ & $i$ & $x_p$ & $ix_p$ \\
$\bf i$ & $i$ & $-1$ & $ix_p$ & $-x_p$ \\
$\bf x_p$ & $x_p$ & $\pm ix_p$ & $*$ & $*$ \\ 
$\bf ix_p$ & $ix_p$ & $\pm x_p$ & $*$ & $*$ \\ 
\end{tabular} \eqno $$
where $*$ denotes an arbitrary element in $\mathcal{A}$, and $x_pi= ix_p$ if and only if $1\leq p\leq m$, while $x_pi=-ix_p$ if and only if $m+1\leq p\leq n-1$.
\end{proposition}

Recall that if   $\mathcal A$  is a nonassociative algebra, then  the set $$\mathcal NC (\mathcal A) = \{ x \in \mathcal A\,|\, xy=yx,\,\text{for\, any\,} y \in \mathcal A\} $$ is called  the
{\it commutative nucleus } of $\mathcal A$. If $\mathcal A$ is unital, then $1\in \mathcal NC (\mathcal A)$. 

\begin{proposition} \label{PAnuc}
Let $\mathcal{A}$ be a four-dimensional partially alternative unital RDA, and let $\mathcal{I_A}$ be its set of imaginary units. Then $\mathcal{I_A} \cap \mathcal NC (\mathcal A) $ is the empty set, that is, there is no imaginary unit commuting
with every element from $\mathcal A$.
\end{proposition}

\begin{proof}
Assume that there exists  $i \in \mathcal{I_A} \cap \mathcal NC (\mathcal A) .$ Denote $\text{span}_{\mathbb R}\{1, i\}$ by $\mathcal C$. Since $\mathcal C$ is a two-dimensional unital RDA, $\mathcal C$ is isomorphic to the algebra of the complex numbers. Choose any $x\in \mathcal A \setminus \mathcal C$.    Then $\{1,i,x,ix\}$   forms a basis for $\mathcal A$.  Then, we  write $x^2=a1+bi+cx+d(ix)$ where $a, b, c, d \in \mathbb R$.
Consider  $y = x-(c/2)1-(d/2)i$ in $\mathcal A$. Let us show that $y^2$ is contained in $\mathcal{C}$. Indeed, we have
\begin{align*}
y^2 &=(x-(c/2)1-(d/2)i)^2  \\
&=x^2+(c^2/4-d^2/4)1+(cd/2)i-cx-d(ix) \\
&=(a+c^2/4-d^2/4)1+(b+cd/2)i
\end{align*}

Since $\mathcal{A}$ is division, we have $y^2\neq 0$. It is clear that  we can find an element $\lambda\in\mathcal{C}$ such that $(\lambda^2)(y^2)=-1$. Set $j =\lambda y$. This  is also an imaginary unit satisfying $(\lambda y)^2=-1$. 

However, as follows from the existence of two different commuting imaginary units that $\mathcal{A}$ is not division. Indeed, $(i+j)(i-j)=i^2-j^2+[j,i]=0$. The proof is complete.
\end{proof}

\begin{corollary}
Let $\mathcal{A}$ be a four-dimensional partially alternative unital RDA, and let $\mathcal{I_A}$ be its set of imaginary units. Assume that there exists at least one $i \in \mathcal{I_A}$. Then we can find a basis of $\mathcal A$: $ \{1, i, j, ij\}$ with respect to which the multiplication table of $\mathcal A$
has the form:
$$\begin{tabular}{c|cccc} 
& $\bf 1$ & $\bf i$ & $\bf  j$ & $\bf ij$ \\ 
\hline
$\bf 1$ & $1$ & $i$ & $j$ & $ij$ \\
$\bf i$ & $i$ & $-1$ & $ij$ & $-j$ \\
$\bf j$ & $j$ & $ -ij$ & $*$ & $*$ \\ 
$\bf ij$ & $ij$ & $ j$ & $*$ & $*$ \\ 
\end{tabular} \eqno $$
where $*$ denotes an arbitrary element in $\mathcal{A}.$
\end{corollary}

\begin{proof}
By Corollary \ref{ESD}, we have that $\mathcal A =  \mathcal{E}_-(i) \oplus \mathcal{E}_+(i)$ and both eigenspaces are even-dimensional.  Note that $1, i \in \mathcal{E}_-(i)$ and, therefore, $\dim_{\mathbb R} \, \mathcal{E}_-(i) \geq 2$.
Assume that $\dim_{\mathbb R}\, \mathcal{E}_-(i)=4$. It follows that $\mathcal A =  \mathcal{E}_-(i)$. By Proposition \ref{GeneralPA} we can choose a basis from $\mathcal{E}_-(i)$: $\{ 1, i, j, ij \}$ with the following multiplication table:
$$\begin{tabular}{c|cccc} 
& $\bf 1$ & $\bf i$ & $\bf  j$ & $\bf ij$ \\ 
\hline
$\bf 1$ & $1$ & $i$ & $j$ & $ij$ \\
$\bf i$ & $i$ & $-1$ & $ij$ & $-j$ \\
$\bf j$ & $j$ & $ ij$ & $*$ & $*$ \\ 
$\bf ij$ & $ij$ & $ -j$ & $*$ & $*$ \\ 
\end{tabular} \eqno $$
where $*$ denotes an arbitrary element in $\mathcal{A}.$ As clear from the above multiplication table, $i$ commutes with every basis element, and, therefore, $i\in \mathcal NC (\mathcal A)$ which is impossible by Proposition \ref{PAnuc}. Consequently,  $\dim_{\mathbb R}\, \mathcal{E}_-(i)=2.$ Then, by Proposition \ref{GeneralPA}, we can find a basis $\{1, i, j, ij\}$ where $1, i \in \mathcal{E}_-(i)$ 
and 
$j, ij \in \mathcal{E}_+(i)$ whose multiplication table has the required form.

The proof is complete.
\end{proof}

\section{Divisibility Conditions for 4-dimensional Partially Alternative  Algebras}

In the previous section, we studied partially alternative algebras of any even dimension $>2$  assuming that at least one imaginary unit exists.



We now turn to questions directly related to the structure of $\mathcal{A}$ given by the multiplication table $(T_p)$. An immediate question is whether the real parameters $a, b, c, d, e, f, g, h$ alone suffice to decide when the algebra is a division algebra. A further question is whether one can deduce, from these coefficients, the invariants of the set of imaginary units $\mathcal{I_A}$, including their geometric configuration in $\operatorname{span}_{\mathbb{R}}(\{i,v,w\})$. As observed in \cite{HNT}, in the absence of the division property, the set of imaginary units may lie on various quadratic surfaces, such as hyperboloids or ellipsoids, making the problem essentially one of solving polynomial equations. For this reason, we restrict our attention to the class of division algebras.

The answers to both questions rely on the quadratic forms. To address the first question, we introduce two quadratic forms in the variables $(\gamma,\delta)\in\mathbb R^2$:
$$p(\gamma,\delta)=(-a-d)\gamma^2+(b-c-e-h)\gamma\delta+(f-g)\delta^2$$
$$q(\gamma,\delta)=(ad-bc)\gamma^2+(ah+de-bg-cf)\gamma\delta+(eh-fg)\delta^2$$

The following proposition was established in \cite{GNT}; however, for completeness, we provide a somewhat different proof below.

\begin{proposition}
Let $\mathcal{A}$ be a four-dimensional algebra given by the multiplication table $(T_p)$ and let $p,q$ be the quadratic forms defined above. 

Then $\mathcal{A}$ is a division algebra if and only if $q(\gamma,\delta)$ is positive definite, and for all $(\gamma,\delta)\neq(0,0)$, the strict inequality $p(\gamma,\delta)> -2\sqrt{q(\gamma,\delta)(\gamma^2+\delta^2)}$ holds.
\end{proposition}
\begin{proof}
We write an arbitrary element of $\mathcal{A}$ as $x=\alpha1+\beta i+\gamma w+\delta v$ and denote by $L_x$ the operator of left multiplication by $x$. With respect to the ordered basis $\{1,i,w,v\}$, a straightforward calculation shows that
$$\det(L_x)=(\alpha^2+\beta^2)^2+p(\gamma,\delta)(\alpha^2+\beta^2)+q(\gamma,\delta)(\gamma^2+\delta^2)$$
where the forms $p$ and $q$ are as defined above.

Clearly, $\mathcal{A}$ is a division algebra if and only if $\det(L_x)=0$ implies $x=0$. We aim to characterize division condition in terms of the parameters by the properties of the two quadratic forms.

We claim that if $\mathcal{A}$ is a division algebra, then $q$ is a positive definite form.  Assume that $q$ is not positive definite and consider the following two possibilities for $q$.

Let $q$ be indefinite. Then, there exists a nonzero vector $(\gamma,\delta)$ such that $q(\gamma,\delta)=0$. Setting $\alpha=\beta=0$, one obtains $\det(L_x)=0$ for $x=\gamma w+\delta v$, contradicting the assumption that $\mathcal{A}$ is division. 

Let $q$ be negative definite. Then, for any nonzero vector $(\gamma,\delta)$ the monic quadratic equation in $\alpha^2+\beta^2$
$$\det(L_x)=(\alpha^2+\beta^2)^2+p(\gamma,\delta)(\alpha^2+\beta^2)+q(\gamma,\delta)(\gamma^2+\delta^2)=0$$
has discriminant
$$\Delta=p(\gamma,\delta)^2-4q(\gamma,\delta)(\gamma^2+\delta^2)>p(\gamma,\delta)^2$$
and hence admits a positive root
$$\dfrac{-p+\sqrt{\Delta}}{2}>0$$
again contradicting the assumption that $\mathcal{A}$ is division. Thus, $q$ is positive definite.

With $q(\gamma,\delta)$ being positive definite, the determinant $\det(L_x)$ is a monic convex quadratic function in $\alpha^2+\beta^2$. It remains strictly positive for $(\alpha^2+\beta^2)\in(0,+\infty)$ if and only if the quadratic has no real roots in $(0,+\infty)$, which occurs precisely when for every $(\gamma,\delta)\neq(0,0)$, at least one of $$p(\gamma,\delta)\geq 0$$ and $$4q(\gamma,\delta)(\gamma^2+\delta^2)> p(\gamma,\delta)^2$$ holds.

Equivalently, using strict positivity of $q(\gamma,\delta)$, the aforementioned condition is equivalent to the following condition: for every $(\gamma,\delta)\neq(0,0)$,
$$p(\gamma,\delta)> -2\sqrt{q(\gamma,\delta)(\gamma^2+\delta^2)}.$$

If now both conditions hold true, then $\det L_x$ takes the zero value only if $x$ is zero. This implies that $\mathcal A$ is a division algebra. 

The proof is complete.

\end{proof}
Every real symmetric $2 \times 2$ matrix has exactly two (real) eigenvalues, counted with multiplicity. For any symmetric matrix $M$, let $\max \sigma(M)$ and $\min \sigma(M)$ denote the largest and smallest eigenvalues of its spectrum, respectively.

Let $P$ and $Q$ denote the matrices associated with the quadratic forms $p(\gamma, \delta)$ and $q(\gamma, \delta)$, as defined above. Then, the following two spectral criteria provide sufficient conditions for the division property.
\begin{corollary}
Let $\mathcal{A}$ be an algebra defined by the multiplication table $(T_p)$. If any  of the following two conditions is satisfied, then $\mathcal{A}$ is a division algebra.
\begin{itemize}
    \item[1.] $q(\gamma,\delta)$ is positive definite and $p(\gamma,\delta)$ is positive semidefinite;
    \item[2.] $q(\gamma,\delta)$ is positive definite and $\min\sigma(P)>-2\sqrt{\min\sigma(Q)}$.
\end{itemize}
\end{corollary}
\begin{proof}

Assume that the first condition holds, that is, $q(\gamma,\delta)$ is positive definite and $p(\gamma,\delta)$ is positive semidefinite. As follows from the proof of the previous proposition, positive definiteness of  $q(\gamma,\delta)$ and the condition $p(\gamma,\delta)\geq 0$ imply that $\mathcal A$ is a division algebra.

Let us now asume that the second condition holds, that is, $q(\gamma,\delta)$ is positive definite and $\min\sigma(P)>-2\sqrt{\min\sigma(Q)}$.

Using Rayleigh quotient, we can bound $p(\gamma,\delta)$ as follows 
$$\max\sigma(P)(\gamma^2+\delta^2)\geq p(\gamma,\delta)\geq\min\sigma(P)(\gamma^2+\delta^2).$$

Likewise, we can also  bound $\sqrt{q(\gamma,\delta)(\gamma^2+\delta^2)}$ 
$$\sqrt{\max\sigma(Q)}(\gamma^2+\delta^2)\geq \sqrt{q(\gamma,\delta)(\gamma^2+\delta^2)} \geq\sqrt{\min\sigma(Q)}(\gamma^2+\delta^2).$$

When $q(\gamma,\delta)$ is positive definite, we have
$$-2\sqrt{\min\sigma(Q)}(\gamma^2+\delta^2)\geq -2\sqrt{q(\gamma,\delta)(\gamma^2+\delta^2)} \geq -2\sqrt{\max\sigma(Q)}(\gamma^2+\delta^2). $$
Combining these inequalities, we obtain 
$$ p(\gamma,\delta)\geq\min\sigma(P)(\gamma^2+\delta^2) \geq -2\sqrt{\min\sigma(Q)}(\gamma^2+\delta^2)\geq -2\sqrt{q(\gamma,\delta)(\gamma^2+\delta^2)}.$$
By the previous proposition, $\mathcal A$ is a division algebra.
\end{proof}

Having settled the preliminary questions concerning the division property, we now examine the multiplication structure. Let 
$\mathcal A$ be a four-dimensional partially alternative real division algebra with multiplication  table $(T_p)$.
We will show that the set of imaginary units $\mathcal{I_A}$ is necessarily one of the following: a discrete set of two elements, a circle, or a sphere. The following discussion provides necessary and sufficient conditions under which the set of imaginary units $\mathcal{I_A}$ is discrete, consisting of exactly $+i$ and $-i$. The remaining cases will be treated uniformly in subsequent sections.

\begin{lemma}
Let $\mathcal{A}$ be a four-dimensional algebra given by the multiplication table $(T_p)$. Then, $\mathcal{I}_\mathcal{A}$ is a subset of $\operatorname{span}_{\mathbb R}(\{i,w,v\})$.
\end{lemma}
\begin{proof}
It is clear that $\pm i\in \operatorname{span}_{\mathbb R}(\{i,w,v\})$. Now, suppose $r1+xi+yw+zv$ is an imaginary unit in $\mathcal{A}$; we wish to prove that $r=0$. Indeed, $(r1+xi+yw+zv)^2$ equals to $s_1 1+s_2 i+2ryw+2rzv$, where $s_1$ and $s_2$ are real-valued functions depending on $r,x,y,z$ and the coefficients $a,b,c,d,e,f,g,h$.

We observe that either $r=0$ or $y=z=0$. In the former case, the conclusion follows immediately. In the latter case, the only possible imaginary units are $\pm i$, which again implies $r=0$. The lemma is established.
\end{proof}

Next we introduce two quadratic forms in the variables $(y,z)\in\mathbb R^2$ to facilitate our discussion:
$$r(y,z)=ay^2+(c+e)yz+gz^2\quad\text{and}\quad s(y,z)=by^2+(d+f)yz+hz^2.$$

\begin{proposition}
Let $\mathcal{A}$ be a four-dimensional RDA given by the multiplication table $(T_p)$ and let $r(y,z), s(y,z)$ be the quadratic forms defined above. 

Then $\mathcal I_\mathcal{A}=\{\pm i\}$, that is, $\mathcal{A}$ admits exactly two imaginary units if and only if $s(y,z)$ is a definite quadratic form.  
\end{proposition}
\begin{proof}
Suppose first that $s(y,z)$ is definite.  Then for any non-zero linear combination $xi+yw+zv\in\operatorname{span}_{\mathbb R}(\{i,w,v\})$ with $(y,z)\neq (0,  0)$ we have that
$$(xi+yw+zv)^2=(-x^2+r(y,z))1+ s(y,z)i\neq -1$$

Now suppose $s(y,z)$ is an indefinite quadratic form. Consider a nonzero vector $(y,z)$ such that the quadratic form $s(y,z)$ vanishes. Then, $(yw+zv)^2= r(y,z)1$. If $r':=r(y,z)$ is nonnegative, then the element $\sqrt{r'}i+yw+zv$ is nonzero and nilpotent of index $2$, contradicting the division property $\mathcal{A}$. Therefore, $r'$ is negative; define $u:=\dfrac{yw+zv}{\sqrt{-r'}}$. Then $u^2=-1$, so $u$ is an imaginary unit distinct from $\pm i$. This proves the claim.
\end{proof}

The next example indicates the existence of a four-dimensional partially alternative real division algebra $\mathcal{A}$ whose set of imaginary units $\mathcal{I_A}$ is $\{\pm i\}$.
\begin{example}
Let $\mathcal{A}$ be a four-dimensional algebra given by the multiplication table $(T_p)$ where $a=d=g=-1$, $b=1$, $c=f=0$, $e=h=1/2$. Then
$$p(y,z)=2y^2+z^2$$
$$q(y,z)=y^2+z^2/4$$
$$s(y,z)=y^2-yz+z^2/2$$
are all positive definite. Hence, $\mathcal{A}$ is a partially alternative real division algebra with exactly two imaginary units.
\end{example}

\section{The Canonical Multiplication Table of Partially Alternative RDAs}

Let $\mathcal A$  denote a four-dimensional partially alternative real division algebra equipped with a unity
1 and a reflection. This section focuses on streamlining the multiplication table $(T_p)$  associated with 
$\mathcal A$ under the assumption of existence of at least three distinct imaginary units. The presence of an additional imaginary unit beyond 
$\pm i$ plays a pivotal role in this analysis: by using its algebraic properties, we systematically reduce 
$(T_p)$ to a simpler form involving only two structure constants. 

\begin{theorem} \label{CanonTable}
Let $\mathcal A$ be a four-dimensional partially alternative unital RDA  that possesses a reflection. Assume that $\mathcal A$ has at least three imaginary units. 
Then there exists a basis $\{1, i, j, k\}$ of $\mathcal A$ with the following multiplication table:

$$ \begin{tabular}{c|cccc} 
          & $\bf 1$            & $\bf i$                                  & $\bf j$                                                 & $\bf  k$      \\ 
\hline
$\bf 1$     &  $1$           & $i$                                  &  $j$                                                    &   $k$     \\
$\bf i$      &  $i$            & $-1$                                &  $-k$                                                  &   $j $ \\
$\bf j$     &  $j$            & $k$                                  &  $-1$                                                &   $-i$   \\ 
$\bf k$     &  $k$           &  $-j$                               &  $ i$                                                  & $ g 1 + hi$  \\ 
\end{tabular} \eqno(T_d) $$
where $g, h\in\mathbb R$. The converse is also true: any four-dimensional RDA with the multiplication table $(T_d)$ having $h\neq 0$  is partially alternative. If $h=0$, then $\mathcal A$ is partially alternative if and only if $g=-1$. In the latter case it is isomorphic to the real quaternion algebra.
\end{theorem}

The proof of this theorem is broken down into a series of lemmas, which follow. To proceed, it is necessary to revisit and summarize the established facts and properties of partially alternative algebras that were obtained in the preceding work \cite{HNT}. 

Recall that if $\mathcal A$ is a four-dimensional 
partially alternative RDA with a unit and a reflection, then $\mathcal A = \mathcal C \oplus \mathcal B$ such that $\mathcal C \cong \mathbb C,$ $\mathcal C\mathcal B =\mathcal B\mathcal C=\mathcal B$ and 
$\mathcal B\mathcal B=\mathcal C$.  We note that $\mathcal C$ contains exactly two imaginary units: $\{-i, i \}$. Hence, if there exists a third one, say $u$, then it does not belong to $\mathcal C$, that is, 
$u \in \mathcal A \setminus \mathcal C.$

We next outline the basis construction for $\mathcal{A}$ with multiplication table $(T_p)$: an initial vector $w \in \mathcal{B}$ is chosen arbitrarily, while the fourth basis vector $v$ is defined as $wi$. 

 This  allows us to choose $w$ in a particular manner. Namely, let us write the third imaginary unit, which exists by the assumption, as
$$u = \lambda_1 1 + \lambda_2 i +\lambda_3 w +\lambda_4 v$$ where $\lambda_i \in \mathbb R$, $i=1, 2, 3, 4.$  By assumption,
 $u \in \mathcal A \setminus \mathcal C;$ therefore, $\lambda^2_3 + \lambda^2_4 \neq 0$.
Then we define a new basis for $\mathcal A$ as follows: $\{1, i, w', v'\}$ where $w' =  \lambda_3 w +\lambda_4 v$ and $v'= w'i$.

As noted above, the multiplication table of $\mathcal{A}$ with respect to a new basis is of type $(T_p)$, potentially involving different structure constants. Additionally, $u$ lies in 
$\mathrm{span}_{\mathbb{R}}\{1, i, w'\}\setminus\mathcal{C}$.


\begin{lemma}
Let $\mathcal A$ be a four-dimensional partially alternative RDA with a unit and a reflection. Let $\{1, i, w, v\}$ be a basis of $\mathcal A$ with multiplication table $(T_p)$.   Assume that there exists $u\in \mathcal{I}_\mathcal{A}\setminus \mathcal C$ and $u\in \text{span}_{\mathbb R}\{1, i, w\}.$
Then after rescaling the basis elements, the multiplication table of  $\mathcal A$  takes the form  $(T_d)$.
\end{lemma}
\begin{proof}
Let $u=x1 + yi +zw$ be an imaginary unit in $\mathcal{I}_\mathcal{A}\setminus \mathcal C$ and $x, y, z \in \mathbb R$. Then routine computation shows that 
$$u^2 =  (x^2 - y^2 + az^2) 1 + (2xy + bz^2) i + (2xz) w.$$
Since $u^2=-1$, we have that 
\begin{align*}
& x^2 - y^2 + az^2 = -1,\\
& 2xy + bz^2 = 0,\\
& 2xz = 0. 
\end{align*}
As follows from the third equation, either $x=0$ or $z=0$. If $z=0$, then $u \in \mathcal C$ which contradicts to our assumption.  
Thus, $z\neq 0$ and $x=0$. This yields the following 
\begin{align*}
& a z^2-y^2  = -1,\\
& bz^2 = 0,
\end{align*}
which implies that $b=0$.

Let us now consider the following possibilities for $a$.

If $a=0$, then $w^2=0$ in $(T_p)$. This contradicts to the assumption that $\mathcal A$ is division. 

If $a>0$, then $$\left( i+\dfrac{w}{\sqrt{a}}\right) \left( i+\dfrac{w}{\sqrt{a}}\right) = i^2 + i \dfrac{w}{\sqrt{a}} +  \dfrac{w}{\sqrt{a}} i +  \left(\dfrac{w}{\sqrt{a}} \right)^2 = -1+ 0 +1 =0, $$   
which is impossible.

Thus, $a<0$. Let us rescale the basis as follows: $w' =  \dfrac{w}{\sqrt{-a}} $ and  $v' =  \dfrac{v}{\sqrt{-a}}.$ Then, the multiplication table of $\mathcal A$ with respect to a new basis $\{1, i, w', v'\}$ takes the form:
$$ \begin{tabular}{c|cccc} 
          & $\bf 1$            & $\bf i$                                  & $\bf w$                                                 & $\bf v$      \\ 
\hline
$\bf 1$     &  $1$           & $i$                                  &  $w$                                                  &   $v$     \\
$\bf i$      &  $i$            & $-1$                                &  $-v$                                                 &   $w $ \\
$\bf w$     &  $w$          & $v$                                 &  $ -1$                                                 & $ c' 1 + d' i$   \\ 
$\bf v$     &  $v$           &  $-w$                              &   $e' 1 + f' i $                                       & $g' 1 +  h' i$  \\ 
\end{tabular} $$
where $c'=\dfrac{c}{a},\, d' =\dfrac{d}{a},\, e'=\dfrac{e}{a},\, f'=\dfrac{f}{a},\, g'=\dfrac{g}{a},\, h'=\dfrac{h}{a}.$
\par\medskip
Further, we recall that $\mathcal A$ is a partially alternative algebra. In particular, we have that $-i=(ww)i = w(wi)$ as $w$ ia also an imaginary unit. Computing $w(wi) = wv = c'1+d'i$ and equating to $-i$, 
we obtain $c'=0$ and $d' =-1$.

Likewise, $-i = i(ww)= (iw)w$. Computing $(iw)w = -vw = -e'1-f'i$ and equating to $-i$,  we have that $e'=0$ and $f'=1$. 

As a result, the multiplication table of $\mathcal A$ takes the required form $(T_d)$.

\end{proof}

\begin{lemma}
Let $\mathcal A$ be a four-dimensional  unital RDA  whose multiplication table is given by $(T_d)$.  If $h\neq 0$, then $\mathcal A$ is partially alternative.
 If $h=0$, then $\mathcal A$ is partially alternative if and only if $g=-1$. In this case $\mathcal A$ is isomorphic to the real quaternion algebra.
\end{lemma}
\begin{proof}
We first assume that $h\neq 0$. Then $\mathcal I_{\mathcal A} = \{ \beta i + \gamma j \,|\, \beta^2 +\gamma^2 = 1\}$. It is a routine check to show that 
for any $u\in \mathcal I_{\mathcal A}$ and $x\in \{i, j, k \}$, we have that $u(ux) =  (xu)u = -x$ and $(ux) u = u(xu)$. Therefore, $\mathcal A$ is indeed  partially 
alternative.

If $h=0$, then $k^2 = g1$. Note that $g$ cannot be zero as $\mathcal A$ is a division algebra. If $g>0$, then 
$$ (\sqrt{g} 1 + k) (\sqrt{g} 1 - k) = g1 - k^2 = g1 - g1 =0$$ which contradicts to the fact that $\mathcal A$ is a division algebra.

If $g<0$, then $u = k/\sqrt{-g}$ is an imaginary unit. Therefore, we must have $u(u i) = u^2 i =-i$. Hence, 
$$ \dfrac{k}{\sqrt{-g}} \left(\dfrac{k}{\sqrt{-g}} i\right) = (-1/g) k(ki) =  (-1/g) (-kj) = (1/g) i =-i$$ implies that $g=-1$, as required. 

On the other hand, it is clear that when $h=0$, $g=-1$, $\mathcal A$ is isomorphic to the real quaternion algebra. The proof is complete.
\end{proof}

\begin{proposition}\label{prop}
Let $\mathcal A$ be a four-dimensional partially alternative algebra with a unit, and let  $\{1, i, j, k\}$ be a basis of $\mathcal A$ with multiplication table $(T_d)$. Then 
$\mathcal A$ is division if and only if $g+h^2/4 < 0$.

\end{proposition}

\begin{proof}
Let $L_x$ denote the operator of left multiplication by $x=\alpha 1+\beta i+\gamma w+\delta v$ where $\alpha, \beta, \gamma, \delta \in \mathbb R$. Then the matrix of $L_x$ with respect to the above 
basis has the form: 
$$
\left( \begin{array}{cccc}
\alpha & -\beta & -\gamma & g\delta \\
\beta  & \alpha & \delta & h\delta - \gamma \\
\gamma & -\delta & \alpha & \beta \\
\delta & \gamma & -\beta & \alpha
\end{array}\right).$$

Routine calculations show that $$\det(L_x) = (\alpha^2 +\beta^2 + \gamma^2 + \delta^2) (\alpha^2 +\beta^2 +\gamma^2 -g\delta^2-h\gamma\delta).$$
Recall that $\mathcal A$ is a division algebra if and only if $\det(L_x)\neq 0$ whenever $x\neq 0$. Notice that $(\alpha^2 + \beta^2 +\gamma^2+\delta^2)\neq 0$ whenever $x\neq 0$.
Therefore, we need to find necessary and sufficient conditions on the parameters $g, h$ that would guarantee $q(\alpha, \beta, \gamma, \delta) = \alpha^2 +\beta^2 +\gamma^2 -g\delta^2-h\gamma\delta $ to be nonzero for $x\neq 0$.  
For this, we re-write the quadratic form $q$ as
$$ q(\alpha, \beta, \gamma, \delta) = \left(\begin{array}{cccc}\alpha&  \beta& \gamma& \delta\end{array}\right) 
\left( \begin{array}{cccc}
1&    0&    0&     0\\
0&    1&    0&     0\\
0&    0&    1&   -h/2 \\
0&    0&    -h/2 & -g
\end{array}\right)
\left(\begin{array}{c}
\alpha\\
\beta\\
\gamma\\
\delta
\end{array}\right).
$$
Clearly, $q$ is a positive-definite quadratic form if and only if the determinate of  $\left(\begin{array}{cc}
                                                                                                                                    1& -h/2\\
                                                                                                                                 -h/2 & -g
                                                                                                                                \end{array}\right)$ is positive, that is, $g+h^2/4 < 0$ as needed.
        \end{proof} 
\par\medskip

\begin{corollary}

Let $\mathcal A$ be a four-dimensional partially alternative unital RDA  that possesses a reflection. Assume that $\mathcal A$ has at least three imaginary units.  Then the following affirmations are equivalent.
\begin{enumerate}
\item $A$ is associative.
\item $A$ is alternative.
\item $A$ is defined by table $(T_d)$ with $h=0$ and $g=-1$.
\item $A$ is isomorphic to the real quaternion algebra.
\end{enumerate}
\end{corollary}

\begin{proof}
As proven, $\mathcal A$ posseses a basis $\{1, i, j, k\}$ such that the multiplication table of $\mathcal A$ takes the form $(T_d)$ for some real $g, h$ satisfying the division condition $g+h^2/4 < 0$.
Hence, the implication $(3) \Rightarrow (2)$ is a straightforward verification. The implications $(2)\Rightarrow (1)$ and $(4)\Rightarrow (3)$ are obvious. 

Let us now show that $(1) \Rightarrow (4)$ holds. Indeed, if $\mathcal A$ is associative, then in particular $k(ij) = (ki) j$ which implies that $k^2 =j^2$. Thus, we have that $g1 + h i = -1$, and $g=-1, h=0$ as needed. The proof is now complete.
\end{proof}

\section{Isomorphism classes and automorphisms of partially alternative RDAs. }

This section aims to classify four-dimensional partially alternative unital RDAs up to isomorphism and characterize their automorphism groups. Classical results  constrain such classifications under strict associativity or alternativity:

1. Frobenius’ theorem isolates $\mathbb{H}$ (quaternions) as the unique 4-dimensional real associative division algebra.

2. Hurwitz’s theorem extends this uniqueness to the alternative case, preserving $\mathbb{H}$ as the only 4-dimensional alternative division algebra.

However, if we replace the strict alternativity condition with the weaker notion of \emph{partial alternativity}, the resulting algebraic structures differ radically from those governed by classical theorems. This relaxation allows for infinitely many non-isomorphic four-dimensional algebras, breaking the uniqueness enforced by Frobenius’ and Hurwitz’s foundational results.
\par\medskip

In what follows let us denote an algebra $\mathcal A$ with multiplication table $(T_d)$ involving structural constants $g, h$ by $\mathcal A_{g,h}$. 
\par\medskip
\noindent{\bf Remark.} 
Let $\mathcal A$ and $\mathcal A'$ be two real unital algebras with ${\mathcal I}_{\mathcal A}$ and ${\mathcal I}_{\mathcal A'}$ being their sets of imaginary units, respectively. Let $\varphi: \mathcal A \to \mathcal A'$ be an algebra homomorphism with $\varphi(1)=1$.  Then $\varphi( {\mathcal I}_{\mathcal A}) \subseteq  {\mathcal I}_{\mathcal A'}.$ Indeed, 
 if $u\in \mathcal I_{\mathcal A}$, then $u^2=-1$. Applying $\varphi$, we obtain $$ \varphi(u)^2=\varphi(u^2)  = \varphi(-1)=-\varphi(1)=1,$$ that is, $\varphi(u)^2 = -1$. Thus, $\varphi(u) \in \mathcal I_{\mathcal A'}$.

\par\medskip

\begin{theorem} \label{isomcond}
Let $A_{g, h}$ and $A_{g', h'}$ be two partially alternative division algebras given by multiplication table $(T_d)$ corresponding to parameters $g, h$ and $g', h'$, respectively.  Then the algebras are isomorphic if and only if $g=g'$ and $h = \pm h'$.
\end{theorem}

\begin{proof}
Assume that  $A_{g, h}$ and $A_{g', h'}$ are isomorphic. Let $\{1, i, j, k\}$ and $\{1, i', j', k'\}$ be two bases of $A_{g, h}$ and $A_{g', h'}$ with respect to which the algebras 
have multiplication tables of type $(T_d)$. 

 Consider the following possibilities. If $h=0$, then by Lemma 3 $g=-1$ and $A_{g, h}$ is the real quaternion algebra. Hence, by isomorphism,
$A_{g', h'}$ is the real quaternion algebra too, and, therefore, $h'=0$, $g'=-1$. Thus, the required condition holds true. 

Consider the case when $h\neq 0$, and, in particular, $\mathcal A_{g, h}$ is not isomorphic to the real quaternion algebra. It follows that $h'\neq 0$ since, otherwise, ${\mathcal A}_{g', h'}$ would be isomorphic to $\mathbb H$.

We observe that $A_{g,h}$ is generated as an algebra by the basis elements $i, j$.

Let $f: A_{g, h}\to A_{g', h'}$ be an algebra isomorphism.  By the above remark, we have that $f({\mathcal I}_{{\mathcal A}_{g, h}}) \subseteq {\mathcal I}_{{\mathcal A}_{g', h'}}$. 
Recall that if $h'\neq 0$, then $ \mathcal I_{{\mathcal A}_{g', h'}} = \{  x i' + y j' \,|\, x^2 + y^2 = 1\}.$ 

This implies that  both $f(i)$ and $f(j)$ belong to $ \mathcal I_{{\mathcal A}_{g', h'}}$, and, hence, 
$f(i) = \beta' i' + \gamma' j'$ and $f(j) = \beta'' i' + \gamma'' j'$ where $\beta'^2 +\gamma'^2 = 1$ and $\beta''^2 +\gamma''^2 = 1$.  

Since $k=ji$, we have that 
\begin{align*}
&f(k) = f(ji) = f(j) f(i) = (\beta'' i' + \gamma'' j'') (\beta' i' + \gamma' j')\\
&= (-\beta'\beta'' - \gamma'\gamma'' )1 + (\beta'\gamma'' - \gamma'\beta'') k'.
\end{align*}

Recall that $k^2 = g1+hi$ in $\mathcal A_{g, h}$. Then applying $f$ to $g1 + hi$ and $k^2$ separately, we obtain: 
\begin{align*}
& f(g1+hi)= g 1 + h f(i) = g1 + h \beta' i' + h \gamma' j' \,\, \text{equals\,to}  \\
& f(k^2)= f(k)^2 =  ( (-\beta' \beta'' - \gamma'\gamma'') 1 + (\beta' \gamma'' - \gamma' \beta'') k' )^2.
\end{align*}

We note that the expression for $f(k^2)$ contains no terms involving the basis element $j'$. This immediately implies that  the term $h \gamma' j'$ in $f(g1+hi)$ must be zero too. Hence, $\gamma'=0$ as $h\neq 0$. Since 
$\beta'^2 +\gamma'^2 = 1$, we have that $\beta'^2=1$. 
Thus, 
\begin{align*}
&f(k^2)= f(k)^2 =  ( (-\beta' \beta'' - \gamma'\gamma'') 1 + (\beta' \gamma'' - \gamma' \beta'') k' )^2\\
&= (\beta')^2 (-\beta'' 1 + \gamma'' k')^2 = (\beta'')^2 - 2 \beta''\gamma'' k' + (\gamma'' )^2 k'^2\\
&=  (\beta'')^2 - 2 \beta''\gamma'' k' + (\gamma'' )^2 (g' 1 + h' i')\\
& = ( ( \beta'')^2 + (\gamma'')^2 g')1 + h' (\gamma'')^2 i'  - 2 \beta''\gamma'' k'.
\end{align*}

Comparing the latter with the expression for  $f(g1+hi)$, we obtain the following equations:

\begin{align*}
 & (\beta'')^2 + (\gamma'')^2 g' = g, \tag{1}\\
 & (\gamma'')^2 h' = h\beta', \tag{2} \\
& \beta'' \gamma'' = 0. \tag{3}
\end{align*}

If $\gamma''=0$, then as follows from the second equation, $\beta'=0$. This means that $f(i) = \beta' i + \gamma' j$ is zero. This is impossible.

Thus, $\beta''=0$. Immediately, $(\gamma'')^2 = 1$.  This yields the required conditions on the parameters: $h' = \pm h$ and $g' =g$. 

Besides, we showed that if $f: A_{g, h}\to A_{g', h'}$ is an algebra isomorphism, then $f(i) = \beta' i'$ and $f(j)= \gamma'' j'$ where $(\beta')^2 = (\gamma'')^2 = 1.$

Conversely, if  the conditions  $g=g'$ and $h=\pm h'$ hold true, then either $A_{g, h}$ and $A_{g', h'}$ have the identical multiplication tables or we can use the mapping $f:\mathcal A_{g, h}\to \mathcal A_{g', h'}$
defined by $f(i)=-i'$ and $f(j)=j'$ to map isomorphically one algebra onto another one.

The proof is complete.

\end{proof}

\begin{corollary}
Let  $\mathcal A_{g, h}$ be a unital partially alternative division algebra with $h\neq 0$, and $f: \mathcal A_{g, h}\to \mathcal A_{g, h}$ be its nontrivial automorphism.
Then $f(i)=i$ and $f(j)=-j.$
\end{corollary}

\begin{proof} As follows from the proof of Theorem \ref{isomcond}, $f(i) = \beta' i$ and $f(j)= \gamma'' j$ where $(\beta')^2 = (\gamma'')^2 = 1.$ 

If $f(i)=-i$, that is, $\beta'=-1$, then 
the equation (2) becomes $ (\gamma'')^2 h = -h$, i.e. $(\gamma'')^2 = -1$, impossible.

Thus, $\beta'=1$, $f(i)=i$. If $f(j)=j$, then this yields the trivial automorphism (the identity automorphism). Hence, the only nontrivial automorphism is given by the conditions  
$f(i)= i$ and $f(j)= -j$. The proof is complete.
\end{proof}

We summarize the automorphism results for $\mathcal A_{g, h}$ in the theorem below.

\begin{theorem} \label{PAaut}
Let $\mathcal A_{g, h}$ be a unital partially alternative division algebra, and $Aut(\mathcal A_{g, h})$ denote its group of automorphisms. Then any automorphism from 
$Aut(\mathcal A_{g, h})$ is inner. Moreover, 

1) If $h=0$, then $\mathcal A_{g, h}$ is isomorphic to the real quaternion algebra, and  $Aut(\mathcal A_{g, h})\cong SO(3)$.

2) If $h\neq 0$, then $Aut(\mathcal A_{g, h}) \cong \mathbb Z_2.$ In this case, the only nontrivial automorphism is a reflection.

\end{theorem}

\section{Lie Algebras Associated With  Partially Alternative Algebras}

In this section, we establish a strengthened version of Proposition 5.2 from \cite{HNT}. The proof is based on the division conditions introduced earlier. Recall that any four-dimensional unital partially alternative division algebra $\mathcal{A}$ endowed with a reflection naturally gives rise to a Lie algebra $\mathcal{L}(\mathcal{A})$: on the underlying vector space of $\mathcal{A}$ we define the Lie bracket by
$$ [x, y] = xy - yx, \quad x, y\in \mathcal A. $$

In \cite{HNT} the following classification result (Proposition 5.2) was proven.

\begin{proposition} \label{Prop5.2}  Let $\mathcal A$ be a four-dimensional partially alternative unital real division algebra with a reflection.
Let $\mathcal L (\mathcal A) = (\mathcal A,\, [\,\,,\,\,])$ be the Lie algebra associated with $\mathcal A$. Let $[v, w]= \alpha 1+\beta i$ where
$\alpha, \beta \in  \mathbb R$.  Write $\mathcal L = \mathbb R 1 + \mathcal I$ where $\mathcal I = \text{Span}_{\mathbb R}\{i, w, v\}$.
\begin{enumerate}
\item If $\alpha=0$ and $\beta \neq 0$, then $\mathcal I$ is a simple Lie ideal isomorphic to $\mathfrak{so}(3)$. Hence,   $\mathcal L \cong \mathfrak{g}_1 \oplus \mathfrak g_{3,7}.$ 
\item If $\alpha=\beta=0$, then $\mathcal L\cong \mathfrak g_1 \oplus \mathfrak g_{3, 5}.$ 
\item If $\alpha \neq 0$ and $\beta \neq 0$, then  $\mathcal L  \cong \mathfrak{g}_1 \oplus \mathfrak g_{3,7}.$ 
\item  If $\alpha\neq 0$ and $\beta = 0$, then $\mathcal L \cong \mathfrak g_{4,9}$ (with zero parameter).
\end{enumerate}
\end{proposition}
\noindent In the above proposition $\mathfrak g_1 \oplus \mathfrak g_{3, 5}$, $ \mathfrak{g}_1 \oplus \mathfrak g_{3,7}$ and $\mathfrak g_{4,9}$ denote the four-dimensional Lie algebras involved in Mubarakzyanov's classification of low-dimensional real Lie algebras \cite{Mub}.
\par\medskip

We next consider two possibilities for $\mathcal A$: either  $\mathcal{I_A}\setminus\mathcal{C}\neq \varnothing$, that is, $\mathcal{A}$ admits at least three imaginary units, or
$\mathcal{I_A}\setminus\mathcal{C} = \varnothing,$ that is, $\mathcal{I_A} = \{\pm i\}$.

\begin{proposition}\label{PropForMoreUnits}
Let $\mathcal{A}$ be a four-dimensional partially alternative real division algebra with a reflection. If $\mathcal{I_A}\setminus\mathcal{C}\neq \varnothing$, that is, $\mathcal{A}$ admits at least three imaginary units, then $ \mathcal L (\mathcal A)\cong  \mathfrak{g}_1 \oplus \mathfrak g_{3,7}.$
\end{proposition}
\begin{proof}
Under the above conditions, Theorem \ref{CanonTable} guarantees a basis $\{1, i, v, w\}$ constructed as in Proposition \ref{Prop5.2}, relative to which the multiplication table of $\mathcal A$ is given by $(T_d)$. It follows that 
$[v, w] = -i$, that is, $\alpha = 0$ and $\beta=-1$. Therefore, by Proposition \ref{Prop5.2},  $\mathcal{L}(\mathcal{A}) \cong \mathfrak{g}_1 \oplus \mathfrak g_{3,7}.$
\end{proof}

\begin{example}
There exists a four-dimensional partially alternative real division algebra $\mathcal{A}$ with a reflection and exactly two imaginary units,  whose associated Lie algebra  $\mathcal L (\mathcal A)$ is isomorphic to $\mathfrak g_1\oplus \mathfrak g_{3,7}$.
\end{example}
\begin{proof}
Set $a=-2$, $b=h=2$, $c=e=0$, $d=g=-1$, and $f=1$ in the multiplication table $(T_p)$. Then, the quadratic forms in question becomes
$$p(\gamma,\delta)=3\gamma^2+2\delta^2,$$
$$q(\gamma,\delta)=2\gamma^2-2\gamma\delta+\delta^2,$$
$$s(y,z)=2y^2+2z^2,$$
which are all positive definite. Hence, we obtain a four-dimensional partially alternative real division algebra $\mathcal{A}$ satisfying $\mathcal{I_A}=\{\pm i\}$. Moreover, $[v,w]=vw-wv=(e-c)1+(f-d)i$. Since $f-d\neq 0$, the associated Lie algebra of $\mathcal{A}$ is isomorphic to $\mathfrak g_1\oplus \mathfrak g_{3,7}$.
\end{proof}

\begin{example}
There exists a four-dimensional partially alternative real division algebra $\mathcal{A}$ with a reflection and exactly two imaginary units, whose associated Lie algebra $ \mathcal L (\mathcal A)$ is isomorphic to $\mathfrak g_{4,9}$ (with parameter zero).
\end{example}
\begin{proof}
Set $a=c=g=-5$, $b=2$, $d=f=-1$, and $e=h=3$ in the multiplication table $(T_p)$. Then, the quadratic forms in question becomes
$$p(\gamma,\delta)=6\gamma^2+\gamma\delta+4\delta^2,$$
$$q(\gamma,\delta)=15\gamma^2-13\gamma\delta+4\delta^2,$$
$$s(y,z)=2y^2-2yz+3z^2,$$
which are all positive definite. Hence, we obtain a four-dimensional partially alternative real division algebra $\mathcal{A}$ satisfying $\mathcal{I_A}=\{\pm i\}$. Moreover, $[v,w]=vw-wv=(e-c)1+(f-d)i$. Since $e-c\neq 0$ and $f-d=0$, the associated Lie algebra of $\mathcal{A}$ is isomorphic to $\mathfrak g_{4,9}$ (with parameter zero).
\end{proof}

\begin{proposition}
Let $\mathcal{A}$ be a four-dimensional partially alternative real algebra with a reflection. If $\mathcal{A}$ is a division algebra, then  $\mathcal L (\mathcal A)$ cannot be isomorphic to $\mathfrak g_1\oplus \mathfrak g_{3,5}$.
\end{proposition}

\begin{remark}
If
$s(y,z)=by^2+2dyz+hz^2$ is definite, then it is clear that $b,h$ have the same sign (i.e. both positive or both negative) and its determinant $\Delta=bh-d^2>0.$
\end{remark}
\begin{proof} By Proposition \ref{PropForMoreUnits}, if $\mathcal{I_A}\setminus\mathcal{C}\neq \varnothing$, then $ \mathcal L (\mathcal A)\cong \mathfrak g_1\oplus \mathfrak g_{3,7}$.
Without loss of generality, we can assume that   $\mathcal{A}$ has exactly two imaginary units, that is, the corresponding $s(y,z)$ is a definite quadratic form, and  $q(\gamma,\delta)$ is positive definite by Propositions 4 and 5.

Under the above conditions let us assume that  $\mathcal{L}(\mathcal{A})\cong \mathfrak g_1\oplus \mathfrak g_{3,5}$. Hence, by Proposition \ref{Prop5.2}, we have that   $0= [v,w]=vw-wv=(e-c)1+(f-d)i$, and, hence,   $c=e$ and $d=f$.

Let $Q$ and $S$ denote the symmetric matrices associated with the quadratic forms $q(\gamma,\delta)$ and $s(y,z)$, respectively.

Let us consider the determinant of  $Q$. It is a routine check to verify that
\begin{align*}
4\det(Q)&=4(ad-bc)(ch-dg)-(ah-bg)^2\\
&=(-h^2)a^2+(4cdh-4d^2g+2bgh)a+(4bcdg-4bc^2h-b^2g^2).
\end{align*}
We next observe that, since $s(y,z)$ is a definite form, the diagonal entry $S_{22}=h$ is nonzero. It follows that the determinant is a quadratic function in $a$.
\begin{align*}
4\det(Q)&=(-h^2)\left(a-\dfrac{4cdh-4d^2g+2bgh}{2h^2}\right)^2\\
&+\dfrac{(4cdh-4d^2g+2bgh)^2}{4h^2}+4bcdg-4bc^2h-b^2g^2\\
&=(-h^2)\left(a-\dfrac{4cdh-4d^2g+2bgh}{2h^2}\right)^2-\dfrac{4(bh-d^2)(ch-dg)^2}{h^2}    .
\end{align*}
Since $4\det(Q)$ is now a concave quadratic, $\det(Q)$ attains its maximum at $a=\dfrac{4cdh-4d^2g+2bgh}{2h^2}$ with corresponding maximum value:  $-\dfrac{(bh-d^2)(ch-dg)^2}{h^2}$.
\par\medskip
Since $s(y,z)$ is a definite form, the determinant of the associated matrix $\det(S)=bh-d^2$ is positive by Remark 3 and the diagonal entry $S_{22}=h$ is nonzero. Since $q(\gamma,\delta)$ is a positive definite form, 
the diagonal entry $Q_{22}=ch-dg$ is nonzero. Therefore, the maximum of the determinant is negative, which contradicts the assumption of the positive-definiteness of $q(\gamma,\delta)$. This completes the proof.
\end{proof}

\par\medskip

In conclusion, there are  a few open questions that invite further investigation.  It remains unknown whether any two of the properties—partial left alternativity, partial flexibility, and partial right alternativity—necessarily entail the third. Besides,  it is unclear whether the assumption that a partially alternative algebra admits a reflection can be omitted. This leads to the broader conjecture that every four-dimensional real division algebra with a non-trivial automorphism group must possess a reflection, for which no counterexamples are presently known. 

The classification of 4-dimensional partially alternative real division algebras  can be described in terms of the existence of reflections and the structure of their sets of imaginary units. More precisely, 

1. If $\mathcal A$  does not admit a reflection (in particular, if its automorphism group is trivial), then the structure of the set of imaginary units is poorly understood, and little is known about the algebraic properties in this case.
\par\medskip
2. If $\mathcal A$ admits a reflection, the classification further divides into the following subcases, depending on the number of imaginary units:
\par\medskip
$\bullet$ If $\mathcal A$ has exactly two imaginary units, then its Lie algebra is either of type $g_1\oplus g_{3,7}$ or $g_{4,9}.$ We plan to address this case  in our subsequent work. 
\par\medskip
$\bullet$ If  $\mathcal A$ has more than two imaginary units, then this case is completely solved in the present paper.

\par\medskip 
\noindent Once the first bullet point is solved, the classification of 4-dimensional partially alternative real division algebras admitting a reflection will be finalized.

\par\medskip
Beyond the problems discussed above, the referee has also suggested several further directions that naturally arise from our results. First, since $B = A^{\ast} = A \setminus \{0\}$ is a loop, it is of independent interest to determine which loop identities $B$ satisfies and how these reflect the underlying partially alternative structure of $A$. Second, one may approach $B$ as an analytic loop and study its tangent algebra, thereby linking our classification to the local differential geometry of nonassociative division loops. Finally, a more global perspective would come from describing the multiplication groups $\mathrm{Mult}(B)$, $\mathrm{Mult}_R(B)$, and $\mathrm{Mult}_L(B)$, generated respectively by $\{L_x, R_x \mid x \in B\}$, $\{R_x \mid x \in B\}$, and $\{L_x \mid x \in B\}$; understanding these groups would clarify how far the symmetries of $B$ extend beyond those captured by the automorphism group of $A$.

\end{document}